\documentclass[12pt,twoside,a4paper]{article}

\usepackage{graphicx} %
\usepackage{amsfonts,amsmath,amsthm,amssymb}
\usepackage{dsfont} %
\usepackage{mathrsfs} %
\usepackage[ps,matrix,arc,arrow]{xy} %

\DeclareMathOperator{\Aut}{Aut} 
\DeclareMathOperator{\Mon}{Mon} 
\DeclareMathOperator{\N}{N} 

\DeclareMathOperator{\D}{D} 

\newtheorem{theorem}{Theorem}[section]

\newtheorem{lemma}[theorem]{Lemma}

\newtheorem{corollary}[theorem]{Corollary}

\def\cA{{\cal A}}
\def\cB{{\cal B}}

\def\cG{{\cal G}}
\def\cH{{\cal H}}

\def\cK{{\cal K}}

\def\cN{{\cal N}}
\def\cO{{\cal O}}
\def\cP{{\cal P}}

\def\cR{{\cal R}}

\def\cT{{\cal T}}
\def\cU{{\cal U}}

\def\im{\rm im}

\title{Constructions of bipartite and bipartite-regular hypermaps}

\author{Rui Duarte\footnote{This work was supported by {\it FEDER} founds through {\it COMPETE}--Operational Programme
Factors of Competitiveness (``Programa Operacional Factores de Competitividade'') and by Portuguese funds through the
{\it Center for Research and Development in Mathematics and Applications} (University of Aveiro) and the Portuguese
Foundation for Science and Technology (``FCT--Funda\c{c}\~{a}o para a Ci\^{e}ncia e a Tecnologia''), within project
PEst-C/MAT/UI4106/2011 with COMPETE number FCOMP-01-0124-FEDER-022690.} \\ Center for Research and Development in
Mathematics and Applications \\
Department of Mathematics \\ University of Aveiro \\ \texttt{rduarte@ua.pt}}

\date{\today}

\begin{document}

\maketitle

\abstract{A hypermap is bipartite if its set of flags can be divided into two parts $A$ and $B$ so that both $A$ and $B$ are the union of vertices, and consecutive vertices around an edge or a face are contained in alternate parts. A bipartite hypermap is bipartite-regular if its set of automorphisms is transitive on $A$ and on $B$.

In this paper we see some properties of the constructions of bipartite hypermaps described algebraically by Breda and
Duarte 
which generalize the construction induced by the Walsh representation of hypermaps. As an application we show that all surfaces have bipartite-regular hypermaps.

}

\section{Definitions and notation}

A \emph{hypermap} $\cH$ is a 4-tuple $(|\cH|, h_0, h_1, h_2)$, where $|\cH|$ is a non-empty finite set and $h_0$, $h_1$ and $h_2$ are involutory permutations on $|\cH|$ generating a transitive group $\Mon (\cH)$, the \emph{monodromy group} of $\cH$. The elements of $|\cH|$ are called \emph{flags} and the orbits of the dihedral subgroups $\langle h_1, h_2 \rangle$, $\langle h_2, h_0 \rangle$ and $\langle h_0, h_1 \rangle$ are called \emph{hypervertices}, \emph{hyperedges} and \emph{hyperfaces} of $\cH$. We use vertices, edges and faces instead of hypervertices, hyperedges and hyperfaces, for short. When $h_2 h_0$ has order 2, $\cH$ is a \emph{map}. Topologically, a hypermap resp. map is a cellular embedding of a connected hypergraph resp. graph into a closed connected surface, usually called the \emph{underlying surface} of the hypermap resp. map.

Given $\cG = (|\cG|, g_0, g_1, g_2)$ and $\cH = (|\cH|, h_0, h_1, h_2)$ hypermaps, a covering from $\cG$ to $\cH$ is a function $\psi: |\cG| \to |\cH|$ such that $g_i \psi = \psi h_i$ for all $i \in \{ 0, 1, 2 \}$. When $\psi$ is one-to-one, $\psi$ is called an \emph{isomorphism}. If there is a covering $\psi$ from $\cG$ to $\cH$, we say that $\cG$ \emph{covers} $\cH$ or that $\cH$ \emph{is covered by} $\cG$, and write $\cG \to \cH$; if $\psi$ is an isomorphism, we say that $\cG$ is isomorphic to $\cH$ or that $\cG$ and $\cH$ are isomorphic and write $\cG \cong \cH$. An \emph{automorphism} or \emph{symmetry} of $\cH$ is an isomorphism from $\cH$ to itself. The set of automorphisms of $\cH$ is denoted by $\Aut (\cH)$.

For any hypermap $\cH$ we have a transitive permutation representation $\Delta \to \Mon (\cH)$ of the free product
$$
\Delta = \langle R_0, R_1, R_2 \mid {R_0}^2 = {R_1}^2 = {R_2}^2 = 1 \rangle \cong C_2 * C_2 * C_2.
$$
The stabilizer $H$ of a flag under the action of $\Delta$, unique up to conjugacy, is called a \emph{hypermap subgroup} for $\cH$. It is well known that hypermap subgroups completely describe hypermaps.

Let $\Theta$ be a normal subgroup of $\Delta$. Following \cite{Breda2006}, $\cH$ is said to be \emph{$\Theta$-conservative} if $H \leq \Theta$, and \emph{$\Theta$-regular} if $H$ is normal in $\Theta$.
A $\Delta$-regular hypermap is best known as \emph{regular}. The \emph{even word subgroup} of $\Delta$ is the subgroup
$$
\Delta^+ = \langle R_1 R_2, R_2 R_0, R_0 R_1 \rangle
$$
which is one of the seven normal subgroups of index 2 in $\Delta$ \cite{BredaJones2000}. The others are
$$
\begin{array}{lll}
\Delta^{\hat 0} = \langle R_1, R_2 \rangle^\Delta, & \Delta^{\hat 1} = \langle R_2, R_0 \rangle^\Delta, & \Delta^{\hat 2} = \langle R_0, R_1 \rangle^\Delta, \\
\Delta^0 = \langle R_0, R_1 R_2 \rangle^\Delta, & \Delta^1 = \langle R_1, R_2 R_0 \rangle^\Delta, & \Delta^2 = \langle R_2, R_0 R_1 \rangle^\Delta.
\end{array}
$$
Usually $\Delta^+$-conservative and $\Delta^+$-regular hypermaps are called \emph{orientable} and \emph{orientably-regular}. An orientably-regular hypermap which is not regular is called \emph{chiral}. Following \cite{Breda2006,BredaDuarte2007}, a hypermap is \emph{bipartite} if it is $\Delta^{\hat 0}$-conservative, and \emph{bipartite-regular} if it is \mbox{$\Delta^{\hat 0}$-regular}.  We say that $\cH$ \emph{has no boundary} if no conjugate of $R_0$, $R_1$ or $R_2$ in $\Delta$ belongs to $H$. Because $R_0 \notin \Delta^{\hat 0}$, a bipartite hypermap has no boundary if and only if no conjugate of $R_1$, $R_2$, ${R_1}^{R_0}$ or ${R_2}^{R_0}$ in $\Delta^{\hat 0}$ belongs to $H$. A hypermap which is not orientable and has no boundary is \emph{non-orientable}. 
Topologically, a hypermap is orientable if its underlying surface is orientable, has no boundary if its underlying surface has no boundary, and is bipartite if we can divide its set of vertices in two parts in such a way that consecutive vertices around an edge or a face are in alternate parts.

Henceforth, we use $|\cH|$ to denote the number of flags of $\cH$ instead of its set of flags, for simplicity. The numbers of vertices, edges and faces of $\cH$ are denoted by $V (\cH)$, $E (\cH)$ and $F (\cH)$. The Euler-Poincar\'{e} characteristic of $\cH$ is the integer
$$
\chi (\cH) = V (\cH) + E (\cH) + F (\cH) - \frac{|\cH|}{2}
$$
and the \emph{genus} of $\cH$, $g (\cH)$, is $(2-\chi(\cH))/2$ or $2-\chi(\cH)$ depending on whether $\cH$ is orientable or not.

The \emph{valency} of a vertex, edge or face of a hypermap without boundary is half its number of flags. A hypermap $\cU$ is \emph{uniform} if there are positive integers $\ell$, $m$ and $n$ such that all vertices have valency $\ell$, all edges have valency $m$ and all faces have valency $n$. The triple $(\ell,m,n)$ is called the \emph{type} of $\cU$.
Regular hypermaps are uniform but, in general, bipartite-regular hypermaps are not uniform. However, given a biparte-regular hypermap $\cB$, there are positive integers $\ell_1$, $\ell_2$, $m$ and $n$ (not necessarily distinct) such that all edges have valency $m$, all faces have valency $n$ and consecutive vertices around an edge or a face have alternate valencies $\ell_1$ and $\ell_2$. The quadruple $(\ell_1, \ell_2; m; n)$ is a \emph{bipartite-type} of $\cB$ (the other is $(\ell_2, \ell_1; m; n)$, if $\ell_1 \neq \ell_2$).

For any $\sigma \in S_{\{ 0, 1, 2 \}}$ 
we have an automorphism $\overline{\sigma}$ of $\Delta$ defined by $R_i \overline{\sigma} = R_{i \sigma}$, for $i \in \{ 0, 1, 2 \}$. This automorphism induces an operation $\D_\sigma$ on hypermaps by assigning to $\cH$ the hypermap $\D_\sigma (\cH)$ with hypermap subgroup $H \overline{\sigma}$, called \emph{$\sigma$-dual} of $\cH$.

Throughout the last century many authors contributed to the classification of regular and chiral maps on surfaces of low genus. More recently, Conder \cite{ConderWWW} obtained lists of regular and chiral maps and hypermaps 
whose Euler-Poincar\'{e} characteristic is negative and greater than or equal to -200, up to isomorphism and duality, using the ``LowIndexNormalSubgroups'' routine in MAGMA \cite{Magma}. Bipartite-regular hypermaps are classified on the sphere \cite{BredaDuarte2007}, the projective plane, the torus and the double torus \cite{Rui2007}.

Let $\cH^+$ be the hypermap with hypermap subgroup $H^+ := H \cap \Delta^+$. Obviously, $\cH^+$ is isomorphic to $\cH$ if $\cH$ is orientable. When $\cH$ is not orientable, $\cH^+$ is a covering of $\cH$ with twice the number of flags usually called the \emph{orientable double cover} of $\cH$.

The \emph{covering core} $\cH_\Delta$ of $\cH$ and the \emph{closure cover} $\cH^\Delta$ of $\cH$ are the hypermaps whose hypermap subgroups are the core $H_\Delta$ of $H$ in $\Delta$ and the normal closure $H^\Delta$ of $H$ in $\Delta$, respectively. While the covering core of $\cH$ is the smallest regular hypermap covering $\cH$, the closure cover of $\cH$ is the largest regular hypermap covered by $\cH$.

When $\cH$ is a $\Theta$-conservative resp. $\Theta$-regular hypermap, $\cH^+$, $\cH^\Delta$ and $\cH_\Delta$ are also $\Theta$-conservative resp. $\Theta$-regular. In fact, both $\cH$ and $\cH^\Delta$ are $\Theta$-conservative resp. \mbox{$\Theta$-regular} or neither is.

\begin{lemma} \label{odc_cc_cc}
For every hypermap $\cH$, $(\cH^+)^\Delta \to (\cH^\Delta)^+$ and $(\cH^+)_\Delta \cong (\cH_\Delta)^+$.
\end{lemma}




More generally, we have the following coverings:

$$
\xymatrix{
(\cH_\Delta)^+ \cong (\cH^+)_\Delta \ar[r] \ar[d] & \cH^+ \ar[r] \ar[d] & (\cH^+)^\Delta \ar[r] & (\cH^\Delta)^+ \ar[d] \\
\cH_\Delta \ar[r] & \cH \ar[rr] & & \cH^\Delta
}
$$

When $\cH$ is $\Theta$-regular, for some normal subgroup $\Theta$ of index 2 in $\Delta$, but not orientable,
$|(\cH^+)^\Delta|$ is equal to $2|(\cH^\Delta)^+|$ or $|(\cH^\Delta)^+|$ depending on whether $\cH_\Delta$ is orientable or
not. Non-orientable bipartite-regular hypermaps whose covering core is orientable can be found on the projective plane
(Chapter 3 of \cite{Rui2007}) and on the Klein bottle (see Section \ref{section_other}).

\section{Constructions of bipartite hypermaps} \label{section_constructions}

In this section we see some properties of the constructions of bipartite hypermaps described in \cite{BredaDuarte2007,Rui2007}.
These constructions are induced from group epimorphisms from $\Delta^{\hat 0}$ to $\Delta$ and generalize the
correspondence between hypermaps and bipartite maps presented by Walsh in \cite{Walsh1975}.
Unless otherwise stated, we shall assume throughout this paper that $\varphi$ is an epimorphism from $\Delta^{\hat 0}$ to $\Delta$.

Let $\cH$ be a hypermap with hypermap subgroup $H$. For every conjugate $H^g$ of $H$ in $\Delta$, $H^g \varphi^{-1}$ is a
conjugate of $H \varphi^{-1}$ in $\Delta^{\hat 0}$ and hence in $\Delta$. In other words, $H^g \varphi^{-1}$ and $H
\varphi^{-1}$ are hypermap subgroups for the same hypermap. Because of this, $\varphi$ induces an operator $\varphi^*$
between the class of hypermaps and the class of bipartite hypermaps. By abuse of language we speak of $\varphi (\cH)$
meaning the hypermap $\varphi^* (\cH)$ with hypermap subgroup $H \varphi^{-1}$. In \cite{BredaDuarte2007,Rui2007}, the
hypermap $\varphi (\cH)$ was denoted by $\cH^{\varphi^{-1}}$. The following result lists some basic properties of these
constructions of bipartite hypermaps.

\begin{lemma} \label{lemma_varphi} Let $\cH$ and $\cG$ be hypermaps and $\Theta$ a normal subgroup of $\Delta$. Then:
\begin{enumerate}
\item $\varphi (\cH)$ has twice the number of flags of $\cH$.
\item $\varphi (\cH)$ is bipartite-regular if and only if $\cH$ is regular. \\
Moreover, $\varphi (\cH)$ is $\Theta \varphi^{-1}$-regular if and only if $\cH$ is $\Theta$-regular.
\item If $\cH$ covers $\cG$, then $\varphi (\cH)$ covers $\varphi (\cG)$.
\item If $\cH$ is isomorphic to $\cG$, then $\varphi (\cH)$ is isomorphic to $\varphi (\cG)$.
\item If $\varphi (\cH)$ is regular, then $\cH$ is regular. \\
More generally, if $\varphi (\cH)$ is $\Theta$-regular, then $\cH$ is $(\Theta \cap \Delta^{\hat 0}) \varphi$-regular.
\item $\Aut (\cH)$ is either $\Aut (\varphi (\cH))$ or a subgroup of $\Aut (\varphi (\cH))$ of index $2$.
\end{enumerate}
\end{lemma}

\begin{proof}
Let $H$ and $G$ be hypermap subgroups of $\cH$ and $\cG$.

\noindent 1. Because $\varphi$ is an epimorphism, $(\Delta^{\hat 0} : H \varphi^{-1}) = (\Delta : H)$ and hence
$$
|\varphi (\cH)| = (\Delta : H \varphi^{-1}) = (\Delta : \Delta^{\hat 0}) (\Delta^{\hat 0} : H
\varphi^{-1}) = 2 (\Delta : H) = 2 |\cH|.
$$
\noindent  2. Since $\varphi$ is an epimorphism, $H \varphi^{-1} \lhd \Theta \varphi^{-1}$ if and only if $H \lhd \Theta$. 

\noindent 3. Let $H \subseteq G^s$ and $s = t \varphi$, for some $t \in \Delta^{\hat 0}$, then $H \varphi^{-1} \subseteq G^s \varphi^{-1} = G^{t \varphi} \varphi^{-1} = (G \varphi^{-1})^t$.

\noindent 4. Follows from the proof of 3 by replacing ``$\subseteq$" with ``$=$".

\noindent 5. When $H \varphi^{-1}$ is normal in $\Theta$, it is also normal in $\Theta \cap \Delta^{\hat 0}$, and so $H = H \varphi^{-1} \varphi$ is normal in $(\Theta \cap \Delta^{\hat 0}) \varphi$.


\noindent 6. Given $g \in \Delta^{\hat 0}$,
$
H = H^{g \varphi} \Leftrightarrow H \varphi^{-1} = (H^{g \varphi})\varphi^{-1} \Leftrightarrow H \varphi^{-1} = (H \varphi^{-1})^g
$
since $\varphi$ is an epimorphism and $(H^{g \varphi})\varphi^{-1} = (H \varphi^{-1})^g$ (see 1 of Lemma 8 of \cite{BredaDuarte2007}, for instance). Consequently,
$
(\N_\Delta (H)) \varphi^{-1} = \N_\Delta (H \varphi^{-1}) \cap \Delta^{\hat 0},
$
and so
$$
\Aut (\cH) \cong \N_\Delta (H) / H \cong (N_\Delta (H)) \varphi^{-1} / H \varphi^{-1} \cong (\N_\Delta (H \varphi^{-1}) \cap \Delta^{\hat 0}) / H \varphi^{-1}.
$$
Because of this, $\Aut (\cH)$ is either $\Aut (\varphi (\cH))$ or a subgroup of $\Aut (\varphi (\cH))$ of index 2 depending on whether $\N_\Delta (H \varphi^{-1})$ is contained in $\Delta^{\hat 0}$ or not.
\end{proof}

The first 4 properties of Lemma \ref{lemma_varphi} are just a reformulation of Lemma 1.6.1 of \cite{Rui2007}, 2 generalizes Theorems 10 and 12 of \cite{BredaDuarte2007}, and 5 generalizes Theorem 10 of \cite{CornSing1988}.

When $\cB$ is a bipartite hypermap such that $\cB \cong \varphi (\cG)$, for some hypermap $\cG$, we say that $\cB$ is constructed using $\varphi$. We denote the class of all bipartite hypermaps constructed using $\varphi$ by $\im \varphi$.
The following result gives us a condition for seeing if a hypermap belongs to $\im \varphi$ or not.

\begin{lemma} \label{constructed_by_phi}
Let $\cB$ be a bipartite hypermap with hypermap subgroup $B$. Then $\cB \in \im \varphi$ if and only if $\ker \varphi \subseteq B^g$, for some $g \in \Delta$. Because $\ker \varphi$ is a normal subgroup of $\Delta^{\hat 0}$, $\cB \in \im \varphi$ if and only if $\ker \varphi \subseteq B$ or $(\ker \varphi)^{R_0} \subseteq B$.
\end{lemma}

\begin{proof}
If $\cB \cong \varphi (\cH)$, for some hypermap $\cK$ with hypermap subgroup $K$, then there is $g \in \Delta$ such that $B^g = K \varphi^{-1}$ and hence $\ker \varphi = 1 \varphi^{-1} \subseteq K \varphi^{-1} = B^g$.

Conversely, assume that $\ker \varphi \subseteq B^g$, for some $g \in \Delta$. Then $B^g \subseteq (\Delta^{\hat 0})^g = \Delta^{\hat 0}$, $B^g = B^g \cdot \ker \varphi = (B^g) \varphi \varphi^{-1}$ and so $\cB \cong \varphi (\cH)$, where $\cH$ is the hypermap with hypermap subgroup $H := (B^g) \varphi$.


Finally, we remark that $\cB \in \im \varphi$ is equivalent to $(\ker \varphi)^{g^{-1}} \subseteq B$ for some $g \in \Delta$, and since $\ker \varphi = 1 \varphi^{-1}$ is a normal subgroup of $\Delta \varphi^{-1} = \Delta^{\hat 0}$, $(\ker \varphi)^{g^{-1}}$ is $\ker \varphi$ or $(\ker \varphi)^{R_0}$ depending on whether $g$ is in
$\Delta^{\hat 0}$ or not.
\end{proof}

\begin{theorem} \label{th_constructed_by_phi}
Let $\cA$ and $\cB$ be bipartite hypermaps.
\begin{enumerate}
\item If $\cA \to \cB$ and $\cA \in \im \varphi$, then $\cB \in \im \varphi$.
\item If $\cB^+ \in \im \varphi$, then $\cB \in \im \varphi$.
\item If $\cB_\Delta \in \im \varphi$, then $\cB \in \im \varphi$.
\item If $\cB \in \im \varphi$, then $\cB^\Delta \in \im \varphi$.
\end{enumerate}
\end{theorem}

\begin{proof}

\noindent 1. Follows from Lemma \ref{constructed_by_phi} and the fact that if $A$ and $B$ are hypermap subgroups for $\cA$ and $\cB$, then $A \subseteq B^h$, for some $h \in \Delta$.

\noindent 2. and 3. Follow from 1 because both $\cB^+$ and $\cB_\Delta$ are bipartite and cover $\cB$.

\noindent 4. Follows from 1 since $\cB^\Delta$ is a bipartite hypermap covered by $\cB$.
\end{proof}

\begin{theorem}
$\varphi (\cH_\Delta)$ is covered by $\varphi (\cH)_\Delta$ and $\varphi (\cH^\Delta)$ covers $\varphi (\cH)^\Delta$.
\end{theorem}

\begin{proof}
%
First of all note that $\varphi (\cH)_\Delta$ and $\varphi (\cH)^\Delta$ are both bipartite since $\varphi (\cH)$ is bipartite. Because $\varphi$ is an epimorphism, $(H \varphi^{-1})_\Delta \varphi$ and $(H \varphi^{-1})^\Delta \varphi$ are normal subgroups of $\Delta$ such that $(H \varphi^{-1})_\Delta \varphi \subseteq H \varphi^{-1} \varphi = H$ and $H = H \varphi^{-1} \varphi \subseteq (H \varphi^{-1})^\Delta \varphi$. Hence, $(H \varphi^{-1})_\Delta \varphi \subseteq H_\Delta$ and $H^\Delta \subseteq (H \varphi^{-1})^\Delta \varphi$.
Finally,
$$
(H \varphi^{-1})_\Delta \subseteq (H \varphi^{-1})_\Delta \ker \varphi = (H \varphi^{-1})_\Delta \varphi \varphi^{-1} \subseteq H_\Delta \varphi^{-1}
$$
and
$$
H^\Delta \varphi^{-1} \subseteq (H \varphi^{-1})^\Delta \varphi \varphi^{-1} = (H \varphi^{-1})^\Delta \ker \varphi = (H \varphi^{-1})^\Delta
$$
because $\ker \varphi = 1 \varphi^{-1} \subseteq H \varphi^{-1} \subseteq (H \varphi^{-1})^\Delta$.
\end{proof}

Hence, we have the following coverings:
$$
\varphi (\cH)_\Delta \to \varphi (\cH_\Delta) \to \varphi (\cH) \to \varphi (\cH^\Delta) \to \varphi (\cH)^\Delta
$$%
In addition, $\varphi (\cH)_\Delta$ is isomorphic to $\varphi (\cH_\Delta)$ if and only if $\varphi (\cH_\Delta)$ is regular, $\varphi (\cH_\Delta)$ and $\varphi (\cH^\Delta)$ are isomorphic to $\varphi (\cH)$ if and only if $\cH$ is regular and $\varphi (\cH^\Delta)$ is isomorphic to $\varphi (\cH)^\Delta$ if and only if $\varphi (\cH^\Delta)$ is regular.

Now we prove a technical result.

\begin{lemma} \label{torsion}
\begin{enumerate}
\item If $g$ is a conjugate of $R_1$, $R_2$, ${R_1}^{R_0}$ or ${R_2}^{R_0}$ in $\Delta^{\hat 0}$, then $g \varphi$ is
either 1 or a conjugate of $R_0$, $R_1$ or $R_2$ in $\Delta$. In addition, $g \in \Delta^+ \varphi^{-1}$ if and only if
$g \in \ker \varphi$.
\item If $h$ is a conjugate of $R_0$, $R_1$ or $R_2$ in $\Delta$, then there is a conjugate $g$ of $R_1$, $R_2$,
${R_1}^{R_0}$ or ${R_2}^{R_0}$ in $\Delta^{\hat 0}$ such that $g \varphi = h$.
\end{enumerate}
\end{lemma}

\begin{proof}
1. Since $(g \varphi)^2 = (g^2) \varphi = 1 \varphi = 1$, $g \varphi$ has order 1 or 2. If $g \varphi$ has order 1, $g
\varphi = 1$ and $g \in \ker \varphi$. When $g \varphi$ has order 2, the torsion theorem for free products (Theorem 1.6
in \S IV.1 of \cite{LyndonSchupp1977}) guarantees that $g \varphi$ is a conjugate of $R_0$, $R_1$ or $R_2$ in $\Delta$.
Because no conjugate of $R_0$, $R_1$ or $R_2$ belongs to $\Delta^+$, the only element of finite order in $\Delta^+$ is 1.
Therefore, $g \varphi \in \Delta^+$ if and only if $g \varphi = 1$, or equivalently, $g \in \Delta^+
\varphi^{-1}$ if and only if $g \in \ker \varphi$.

\noindent 2. Let $S = \{ R_1, R_2, {R_1}^{R_0}, {R_2}^{R_0} \}$, $\{ i, j, k \} = \{ 0, 1, 2 \}$ and $h = {R_k}^c$, for
some $c \in \Delta$. The set $S \varphi$ has a conjugate of $R_k$ in $\Delta$. For otherwise all elements of $S \varphi$
would be either 1 or conjugates of $R_i$ or $R_j$ in $\Delta$, and so $\Delta^{\hat 0} \varphi = \langle S \rangle \varphi
= \langle S \varphi \rangle \subseteq \langle R_i, R_j \rangle^\Delta = \Delta^{\hat k} \neq \Delta$. Let $s \in S$ be such
that $s \varphi = {R_k}^b$, for some $b \in \Delta$. Finally, let $a \in \Delta^{\hat 0}$ be such that $a \varphi = b^{-1}c$.
Then $g := s^a$ is a conjugate of $R_1$, $R_2$, ${R_1}^{R_0}$ or ${R_2}^{R_0}$ in $\Delta^{\hat 0}$ such that $g \varphi = h$.
\end{proof}

\begin{theorem} \label{th_b_o_r}
\begin{enumerate}
\item If $\varphi (\cH)$ has no boundary, then $\cH$ has no boundary.
\item If $\varphi (\cH)$ is orientable, then $\cH$ is orientable.
\end{enumerate}
\end{theorem}

\begin{proof}
Let $H$ be a hypermap subgroup for $\cH$.

\noindent 1. If $\cH$ has boundary, then there is a conjugate $h$ of $R_0$, $R_1$ and $R_2$ such that $h \in H$. By 2.
of Lemma \ref{torsion}, there is a conjugate $g$ of $R_1$, $R_2$, ${R_1}^{R_0}$ or ${R_2}^{R_0}$ in $\Delta^{\hat 0}$
such that $g \varphi = h \in H$. It follows that $g \in H \varphi^{-1}$ and hence $\varphi (\cH)$ has boundary.

\noindent 2. Let $h = h_1 \cdots h_k \in H$, where $h_1, \ldots, h_k \in \{ R_0, R_1, R_2 \}$. By 2. of Lemma
\ref{torsion}, there are conjugates $g_1, \ldots, g_k$ of $R_1$, $R_2$, ${R_1}^{R_0}$ or ${R_2}^{R_0}$ in $\Delta^{\hat
0}$ such that $g_i \varphi = h_i$, for all $i \in \{ 1, \ldots, k \}$. Since $g = g_1 \cdots g_k \in \Delta^{\hat 0}$
and $g \varphi = h \in H$, $g \in H \varphi^{-1}$. In addition, $g \in \Delta^+$ if and only if $k$ is even, that is,
if and only if $h \in \Delta^+$. Consequently, if $H \nsubseteq \Delta^+$, then $H \varphi^{-1} \nsubseteq \Delta^+$,
that is, if $\cH$ is not orientable, then $\varphi (\cH)$ is not orientable.
\end{proof}

We now come to the main result of this section.

\begin{theorem} \label{th_main}
Let $\cH$, $\cN$ and $\cO$ be hypermaps such that $\cN$ has no boundary and $\cO$ is orientable. The following conditions are equivalent:
\begin{enumerate}
\item $R_1, R_2, {R_1}^{R_0}, {R_2}^{R_0} \notin \Delta^+ \varphi^{-1}$.
\item $\Delta^+ \varphi^{-1} = \Delta^+ \cap \Delta^{\hat 0} = \Delta^{+0\hat{0}}$, that is, $\varphi (\cT_{\Delta^+}) \cong \cT_{\Delta^{+0\hat{0}}}$.
\item $\ker \varphi \subseteq \Delta^+$.
\item $\varphi(\cN)$ has no boundary.
\item $\varphi (\cO)$ is orientable.
\item $\varphi (\cH^+) \cong \varphi (\cH)^+$.
\item $\varphi(\cH)^+ \in \im \varphi$.
\end{enumerate}
\end{theorem}

\begin{proof}
Let $H$, $N$ and $O$ be hypermap subgroups for $\cH$, $\cN$ and $\cO$, respectively.

\noindent $(1 \Rightarrow 2)$ %
Let $S := \{ g \in \Delta^{\hat 0} \mid g \in \Delta^+ \Leftrightarrow g \in \Delta^+ \varphi^{-1} \}$. Since:
\begin{itemize}
\item $R_1, R_2, {R_1}^{R_0}, {R_2}^{R_0} \in S \subseteq \Delta^{\hat 0}$, and so $S \neq \varnothing$,
\item for all $g, h \in S$, $gh \in S$, since
$$
\begin{array}{rcl}
gh \in \Delta^+ \varphi^{-1} & \Leftrightarrow & (gh) \varphi = (g \varphi)(h \varphi) \in \Delta^+ \\
& \Leftrightarrow & g \varphi, h \varphi \in \Delta^+ \text{ or } g \varphi, h \varphi \notin \Delta^+ \\
& \Leftrightarrow & g, h \in \Delta^+ \varphi^{-1} \text{ or } g, h \notin \Delta^+ \varphi^{-1} \\
& \Leftrightarrow & g, h \in \Delta^+ \text{ or } g, h \notin \Delta^+ \\
& \Leftrightarrow & g h \in \Delta^+, \\
\end{array}
$$
\item for all $g \in S$, $g^{-1} \in S$, because
$$
g^{-1} \in \Delta^+ \varphi^{-1} \Leftrightarrow g \in \Delta^+ \varphi^{-1} \Leftrightarrow g \in \Delta^+
\Leftrightarrow g^{-1} \in \Delta^+,
$$
\end{itemize}
it follows that $S$ is a subgroup of $\Delta^{\hat 0}$ containing $\Delta^{\hat 0}$, that is, $S = \Delta^{\hat 0}$. Therefore,
$$
g \in \Delta^+ \Leftrightarrow g \varphi \in \Delta^+ \Leftrightarrow g \in \Delta^+ \varphi^{-1},
$$
for all $g \in \Delta^{\hat 0}$, that is, $\Delta^+ \cap \Delta^{\hat 0} = \Delta^+ \varphi^{-1} \cap \Delta^{\hat 0} =  \Delta^+ \varphi^{-1}$.

\noindent $(2 \Rightarrow 5)$ Because $O \varphi^{-1} \subseteq \Delta^{\hat 0}$, $O \subseteq \Delta^+ \Leftrightarrow
O \varphi^{-1} \subseteq \Delta^+ \varphi^{-1} = \Delta^+ \cap \Delta^{\hat 0} \Leftrightarrow O \varphi^{-1} \subseteq
\Delta^+$.

\noindent $(5 \Rightarrow 3)$ $\ker \varphi = 1 \varphi^{-1} \subseteq O \varphi^{-1} \subseteq \Delta^+$.

\noindent $(3 \Rightarrow 1)$ Follows from 1. of Lemma \ref{torsion}, since $R_1, R_2, {R_1}^{R_0}, {R_2}^{R_0} \notin
\Delta^+$.

\noindent $(1 \Rightarrow 4)$ Assume that $R_1, R_2, {R_1}^{R_0}, {R_2}^{R_0} \notin \Delta^+ \varphi^{-1}$. Let $g$ be
a conjugate of $R_1$, $R_2$, ${R_1}^{R_0}$ or ${R_2}^{R_0}$ in $\Delta^{\hat 0}$. Since $\Delta^+ \varphi^{-1}$ is a
normal subgroup of $\Delta^{\hat 0}$, $g \notin \Delta^+ \varphi^{-1}$. In particular, $g \varphi \neq 1 \in \Delta^+$.
By 1. of Lemma \ref{torsion}, $g \varphi$ is a conjugate of $R_0$, $R_1$ or $R_2$ in $\Delta$. Since $\cN$ has no
boundary, $g \varphi \notin N$, and so $g \notin N \varphi^{-1}$. Because of this, no conjugate of $R_1$, $R_2$,
${R_1}^{R_0}$ or ${R_2}^{R_0}$ in $\Delta^{\hat 0}$ is in $N \varphi^{-1}$, that is, $\varphi (\cN)$ has no boundary.

\noindent $(4 \Rightarrow 1)$ Let $g$ be $R_1$, $R_2$, ${R_1}^{R_0}$ or ${R_2}^{R_0}$. Since $\varphi (\cN)$ has no
boundary, $g \notin N \varphi^{-1}$ and hence $g \varphi \neq 1 \in N$. By 1. of Lemma \ref{torsion}, $g \varphi$ is a
conjugate of $R_0$, $R_1$ or $R_2$ in $\Delta$. In particular $g \notin \Delta^+$, or equivalently, $g \notin \Delta^+
\varphi^{-1}$.

\noindent $(2 \Rightarrow 6)$ Since $(H^+) \varphi^{-1} = (H \cap \Delta^+) \varphi^{-1}$ and $(H \varphi^{-1})^+ = (H
\varphi^{-1}) \cap \Delta^+$ are hypermap subgroups for $\varphi (\cH^+)$ and $\varphi (\cH)^+$, and
$$
(H \cap \Delta^+) \varphi^{-1} = H \varphi^{-1} \cap \Delta^+ \varphi^{-1} = H \varphi^{-1} \cap (\Delta^+ \cap
\Delta^{\hat 0}) = (H \varphi^{-1} \cap \Delta^{\hat 0}) \cap \Delta^+ = (H \varphi^{-1}) \cap \Delta^+,
$$
$\varphi (\cH^+)$ and $\varphi (\cH)^+$ are isomorphic.

\noindent $(6 \Rightarrow 7)$ Clearly, $\varphi(\cH^+) \in \im \varphi$.

\noindent $(7 \Rightarrow 3)$ By Lemma \ref{constructed_by_phi} and because $(H \varphi^{-1})^+$ is a hypermap subgroup
for $\varphi(\cH)^+$, $\ker \varphi \subseteq ((H \varphi^{-1})^+)^g = (H \varphi^{-1} \cap \Delta^+)^g \subseteq
(\Delta^+)^g = \Delta^+$, for some $g \in \Delta$.
\end{proof}

As an immediate corollary to Theorem \ref{th_main} (and Theorems \ref{th_constructed_by_phi} and \ref{th_b_o_r}) we get the following.

\begin{corollary} \label{cor_main}
Let $\varphi$ be an epimorphism such that $R_1, R_2, {R_1}^{R_0}, {R_2}^{R_0} \notin \Delta^+ \varphi^{-1}$. Then:
\begin{enumerate}
\item Both $\cH$ and $\varphi (\cH)$ are orientable or neither is.
\item Both $\cH$ and $\varphi (\cH)$ have boundary or neither has.
\item If $\cB$ is bipartite, then both $\cB$ and $\cB^+$ belong to $\im \varphi$ or neither does.
\end{enumerate}
\end{corollary}

\begin{corollary} \label{non-orientable_and_bipartite}
If $\cB \cong \varphi (\cH)$ is a non-orientable hypermap, then $\cB^+ \cong \varphi (\cH^+)$.
\end{corollary}

\begin{proof}
By Theorem \ref{th_b_o_r}, $\cH$ has no boundary. The proof now follows from Theorem \ref{th_main}.
\end{proof}

We now give several examples of epimorphisms from $\Delta^{\hat 0}$ to $\Delta$.

As noted in \cite{BredaDuarte2007}, the Reidemeister-Schreier rewriting process \cite{Johnson1980} can be used to show that
$$
\Delta^{\hat 0} = \langle R_1 \rangle * \langle R_2 \rangle * \langle {R_1}^{R_0} \rangle
* \langle {R_2}^{R_0} \rangle.
$$
Because of this, every epimorphism from $\Delta^{\hat 0}$ to $\Delta$ is completely determined by the images of $R_1$, $R_2$, ${R_1}^{R_0}$ and ${R_2}^{R_0}$. Let $\varphi_1, \varphi_2, \varphi_3, \varphi_4, \varphi_5$ be the epimorphisms from $\Delta^{\hat 0}$ to $\Delta$ whose images of $R_1$, $R_2$, ${R_1}^{R_0}$ and ${R_2}^{R_0}$ are listed in Table \ref{tabela_fi}. The construction induced by $\varphi_1$ is the well known Walsh's correspondence between hypermaps and bipartite maps \cite{Walsh1975}, and $\varphi_2$ induces a construction described in \cite{BredaDuarte2007,Rui2007}. Note that $\varphi_4 = (12) \varphi_3 (12)$.

\begin{table}[hhh]
$$
\begin{array}{|l||l|l|l|l|l|}
\hline g & g \varphi_1 & g \varphi_2 & g \varphi_3 & g \varphi_4 & g \varphi_5 \\
\hline %
\hline R_1 & R_1 & R_1 & R_1 & R_1 & R_1 \\
\hline R_2 & R_2 & R_2 & R_2 & R_2 & R_2 \\
\hline {R_1}^{R_0} & R_0 & R_0 & R_2 & R_0 & {R_1}^{R_0} \\
\hline {R_2}^{R_0} & R_2 & R_0 & R_0 & R_1 & R_0 \\
\hline
\end{array}
$$
\caption{The images of $R_1$, $R_2$, ${R_1}^{R_0}$ and ${R_2}^{R_0}$ by $\varphi_1$, $\varphi_2$, $\varphi_3$,
$\varphi_4$ and $\varphi_5$.} \label{tabela_fi}
\end{table}

Let $\varphi_i$ be one of these epimorphisms. Since $R_1, R_2, {R_1}^{R_0}, {R_2}^{R_0} \notin \Delta^+ {\varphi_i}^{-1}$, the construction induced by $\varphi_i$ has the following properties:
\begin{itemize}
\item both $\cH$ and $\varphi_i (\cH)$ are orientable or neither is,
\item both $\cH$ and $\varphi_i (\cH)$ have boundary or neither has,
\item $\varphi_i (\cH^+) \cong \varphi_i (\cH)^+$,
\item if $\cB$ is bipartite, then both $\cB$ and $\cB^+$ are constructed using $\varphi_i$ or neither is.
\end{itemize}

Although $\cH$, $\varphi_1 (\cH)$ and $\varphi_2 (\cH)$ all have the same underlying surface \cite{BredaDuarte2007}, in general $\cH$ and $\varphi (\cH)$ may not have the same Euler characteristic.
When $\cR$ is a regular hypermap of type $(\ell,m,n)$, $\varphi_3 (\cR)$ and $\varphi_4 (\cR)$ are bipartite-regular hypermaps of bipartite-type $(\ell, m; 2 \ell; 2m)$ and $(\ell, n; 2 \ell; 2n)$, and $\varphi_5 (\cR)$ is a bipartite-regular hypermap of bipartite-type $(\ell, n; 2m; n)$ if $n$ is
even, or $(\ell, n; 2m; 2n)$ if $n$ is odd. Using the Euler formula for bipartite-uniform hypermaps without boundary
(Lemma 19 of \cite{BredaDuarte2007}), we obtain:
$$
\chi (\varphi_3 (\cR)) = 2 \left ( \chi (\cR) - F (\cR) \right), \quad \chi (\varphi_4 (\cR)) = 2 \left ( \chi (\cR) -
E (\cR) \right)
$$
and
$$
\chi (\varphi_5 (\cR)) =
\begin{cases}
\chi (\cR) + 2 F (\cR) - \frac{|\cR|}{2} & \text{if $n$ is even} \\
\chi (\cR) + F (\cR) - \frac{|\cR|}{2} & \text{if $n$ is odd} \\
\end{cases}.
$$

Two interesting questions which naturally occur are whether all bipartite-regular hypermaps are obtained from regular hypermaps using these constructions, and whether there is a surface (without boundary) having no bipartite-regular hypermaps. The answers to both questions are ``no" as we will see in Sections \ref{section_other} and \ref{section_surfaces_with}.

\section{Some bipartite hypermaps} \label{section_other}

In this section we present bipartite hypermaps which are not obtained by the constructions described above.

Let $\cK$ be the bipartite-regular hypermap with hypermap subgroup
$$
K := \langle \{ (uv)^2 \mid u, v \in \{ R_1, R_2, {R_1}^{R_0}, {R_2}^{R_0} \} \} \cup \{ R_2 {R_2}^{R_0} {R_1}^{R_0} \} \rangle^{\Delta^{\hat 0}},
$$
One can see that $\cK$ is a uniform hypermap of type $(2,4,4)$ on the Klein bottle with 16 flags. Its orientable double cover $\cT := \cK^+$ is also bipartite-regular and uniform of type $(2,4,4)$, but is on the torus and has $32$ flags. A hypermap subgroup for $\cT$ is
$$
T := K \cap \Delta^+ = \langle \{ (uv)^2 \mid u, v \in \{ R_1, R_2, {R_1}^{R_0}, {R_2}^{R_0} \} \} \rangle^{\Delta^{\hat 0}}.
$$
In fact, $\cT$ is regular and the dual $\D_{(01)} (\cT)$ is the regular map
denoted by $\{ 4, 4 \}_{2,0}$ in 
\cite{CoxeterMoser1980}, by $\{ i \}_{2}$ in 
\cite{SingSydd2000}, and by $(4,2,4)_{\scriptsize {2 \ 0 \choose 0 \ 2}}$ in 
\cite{Rui2007}. Since $\cK$ is not regular, $\cT$ is the covering core of $\cK$, and hence $(\cK^+)^\Delta$ and
$(\cK^\Delta)^+$ are non-isomorphic.

When $\cB$ is a bipartite-regular hypermap such that $\cB \in \im \varphi$ and $B$ is a hypermap subgroup for $\cB$, then $\Delta^{\hat 0} / B \cong \Delta / B \varphi$, and so $\Delta^{\hat 0} / B$ is generated by up to 3 elements. Since no 3 elements generate $\Delta^{\hat 0} / T \cong C_2 \times C_2 \times C_2 \times C_2$, the hypermap $\cT$ cannot be constructed using an epimorphism $\varphi$. Corollary \ref{non-orientable_and_bipartite} ensures that neither $\cK$ can be constructed using $\varphi$. It also follows that $\Delta^{\hat 0} / B$ being generated by up to 3 elements is a necessary but not sufficient condition for $\cB$ being constructed using an epimorphism.

Let $\cP_2$ be the unique regular hypermap on the sphere of type $(2,2,2)$ whose hypermap subgroup is $\Delta' = \langle (R_1 R_2)^2, (R_2 R_0)^2, (R_0 R_1)^2 \rangle^\Delta$, the derived subgroup 
of $\Delta$. Then $\varphi_2 (\cP_2)$ is a bipartite-regular hypermap of bipartite-type $(1,2;4;4)$ and its covering core $\varphi_2 (\cP_2)_\Delta$ is $\cT$ (see \cite{BredaDuarte2007,Rui2007}), showing that the converse of 3 of Theorem \ref{th_constructed_by_phi} is not true.

\section{Surfaces with bipartite-regular hypermaps} \label{section_surfaces_with}

According to Theorems 10 and 12 of \cite{BredaDuarte2007}, the constructions induced by $\varphi_1$ and $\varphi_2$ guarantee the existence of bipartite-regular hypermaps on every surface which supports regular hypermaps. However, there are non-orientable surfaces having no regular hypermaps, such as the non-orientable surfaces of characteristic $0$, $-1$, $-16$, $-22$, $-25$, $-37$ and $-46$ (see \cite{BredaJones2001b,WilsonBreda2004}). Do these surfaces have bipartite-regular hypermaps? The answer is given by the following result. 

\begin{theorem}
All surfaces have bipartite-regular hypermaps.
\end{theorem}

\begin{proof}
The existence of bipartite-regular hypermaps on every orientable surface is a consequence of Theorem 23 of
\cite{BredaDuarte2007}.

Let $\cP \cP_{2k}$ be the regular hypermap on the projective plane of type $(2,2,2k)$, with $k$ vertices, $k$
edges, $1$ face and $4k$ flags. By Theorem \ref{th_main}, the bipartite-regular hypermaps $\varphi_4 (\cP \cP_{2k})$ and $\varphi_5 (\cP \cP_{2k})$ are both non-orientable.
Since
$$
\chi (\varphi_4 (\cP \cP_{2k})) = 2 - 2k
\qquad \text{and} \qquad
\chi (\varphi_5 (\cP \cP_{2k})) = 2 - (2k-1),
$$
the hypermaps $\varphi_4 (\cP \cP_{2k})$ and $\varphi_5 (\cP \cP_{2k})$ have genus $2k$ and $2k-1$, respectively.
\end{proof}


\bibliographystyle{amsplain}
\bibliography{bibliography}

\providecommand{\bysame}{\leavevmode\hbox to3em{\hrulefill}\thinspace}
\providecommand{\MR}{\relax\ifhmode\unskip\space\fi MR }
\providecommand{\MRhref}[2]{%
  \href{http://www.ams.org/mathscinet-getitem?mr=#1}{#2}
}
\providecommand{\href}[2]{#2}
\begin{thebibliography}{10}

\bibitem{Magma}
Wieb Bosma, John Cannon, and Catherine Playoust, \emph{The {MAGMA} algebra
  system. {I}. the user language}, J. Symbolic Comput. \textbf{24} (1997),
  235--265.

\bibitem{Breda2006}
Antonio Breda~d'Azevedo, \emph{A theory of restricted regularity of hypermaps},
  J. Korean Math. Soc. \textbf{43} (2006), no.~5, 991--1018.

\bibitem{BredaDuarte2007}
Antonio Breda~d'Azevedo and Rui Duarte, \emph{Bipartite-uniform hypermaps on
  the sphere}, Electron. J. Combin. \textbf{14} (2007), 1--20, also available
  at arXiv:math.CO/0607281.

\bibitem{BredaJones2000}
Antonio~J. Breda~d'Azevedo and Gareth~A. Jones, \emph{{Double coverings and
  reflexive abelian hypermaps}}, Beitr. Algebra Geom. \textbf{41} (2000),
  no.~2, 371--389.

\bibitem{BredaJones2001b}
\bysame, \emph{{Rotary hypermaps of genus 2}}, Beitr. Algebra Geom. \textbf{42}
  (2001), no.~1, 39--58.

\bibitem{ConderWWW}
Marston Conder, \emph{{M}arston {C}onder's home page},
  \texttt{http://www.math.auckland.ac.nz/$\sim$conder/}.

\bibitem{CornSing1988}
David Corn and David Singerman, \emph{Regular hypermaps}, European J. Combin.
  \textbf{9} (1988), no.~4, 337--351.

\bibitem{CoxeterMoser1980}
H.S.M. Coxeter and W.O.J. Moser, \emph{Generators and relations for discrete
  groups}, 4th ed., Ergebnisse der Mathematik und ihrer Grenzgebiete, Bd. 14,
  Springer-Verlag, Berlin-Heidelberg-New York, 1980.

\bibitem{Rui2007}
Rui Duarte, \emph{2-restrictedly-regular hypermaps of small genus}, Ph.D.
  thesis, University of Aveiro, 2007.

\bibitem{Johnson1980}
D.~L. Johnson, \emph{Topics in the theory of groups presentations}, London
  Math. Soc. Lecture Note Series, vol.~42, Cambridge University Press, 1980.

\bibitem{LyndonSchupp1977}
R.~C. Lyndon and P.~E. Schupp, \emph{Combinatorial group theory},
  Springer-Verlag, Berlin-Heidelberg-New York, 1977.

\bibitem{SingSydd2000}
David Singerman and Robert~I. Syddall, \emph{Geometric structures on toroidal
  maps and elliptic curves}, Math. Slovaca \textbf{50} (2000), no.~5, 495--512.

\bibitem{Walsh1975}
T.R.S. Walsh, \emph{{Hypermaps versus bipartite maps}}, J. Comb. Theory, Ser. B
  \textbf{18} (1975), 155--163.

\bibitem{WilsonBreda2004}
Steve Wilson and Antonio Breda~d'Azevedo, \emph{{Surfaces having no regular
  hypermaps}}, Discrete Math. \textbf{277} (2004), no.~1-3, 241--274.

\end{thebibliography}

\end{document}